\documentclass[reqno]{amsart}

\usepackage{amsmath, amssymb}
\usepackage[mathscr]{eucal}
\usepackage{hyperref}
\hypersetup{colorlinks=true, linkcolor=blue, urlcolor=blue, citecolor=[rgb]{0, 0.8, 0}, linktocpage=true}

\theoremstyle{plain}
\newtheorem*{theorem}{Theorem}

\newcommand{\du}{\mathrm{d}}

\newcommand{\iu}{\mathrm{i}}
\newcommand{\bbC}{\mathbb{C}}
\newcommand{\bbR}{\mathbb{R}}
\newcommand{\frD}{\mathfrak{D}}
\newcommand{\frH}{\mathfrak{H}}

\DeclareMathOperator{\im}{Im}
\DeclareMathOperator{\linspan}{span}

\title{On extensions of symmetric operators}

\author{Namig J. Guliyev}
\address{Institute of Mathematics and Mechanics, Azerbaijan National Academy of Sciences, 9 B.~Vahabzadeh str., AZ1141, Baku, Azerbaijan.}
\email{njguliyev@gmail.com}



\begin{document}
\maketitle
\begin{abstract}
We give an explicit description of all minimal self-adjoint extensions of a densely defined, closed symmetric operator in a Hilbert space with deficiency indices $(1, 1)$.
\end{abstract}

There is a widely used linearization technique in the theory of Sturm--Liouville problems with boundary and/or discontinuity conditions polynomially dependent on the eigenvalue parameter. One considers a Hilbert (or Pontryagin) space of the form $L^2 \oplus \mathbb{C}^k$ and constructs a self-adjoint operator in this space such that the eigenvalue problem for this operator and the original boundary value problem become equivalent, in the sense that their eigenvalues coincide, the eigenfunctions of the latter problem are in one-to-one correspondence with the first components of the eigenvectors of the former problem, and so on (see, e.g., \cite{BCNW18}, \cite{BP18}, \cite{G17}, and the references therein). Fulton \cite[Remark 2.1]{F77} attributes this technique to Friedman \cite[pp. 205--207]{F56}. The purpose of this short paper is to show that a straightforward generalization of this technique gives an explicit description of all minimal self-adjoint extensions of a densely defined, closed symmetric operator with deficiency indices $(1, 1)$ (see below for definitions).

Let $A$ be a densely defined, closed symmetric operator in a separable Hilbert space $\frH$ with deficiency indices $(1, 1)$. Let $\{ \bbC, \Gamma_0, \Gamma_1 \}$ be a \emph{boundary triplet} for $A^{*}$. This means that $\Gamma_0$, $\Gamma_1 \colon \frD(A^{*}) \to \bbC$ are two linear mappings such that abstract Green's identity
$$
  \langle A^{*} x, y \rangle_{\frH} - \langle x, A^{*} y \rangle_{\frH} = \Gamma_1 x \cdot \overline{\Gamma_0 y} - \Gamma_0 x \cdot \overline{\Gamma_1 y}, \qquad x, y \in \frD(A^{*})
$$
holds and the mapping $\Gamma \colon \frD(A^{*}) \to \bbC^2,\ x \mapsto (\Gamma_0 x, \Gamma_1 x)$ is surjective \cite{D15}, \cite[Chapter 14]{S12}. The domain of $A$ then coincides with the kernel of $\Gamma$:
\begin{equation} \label{eq:kernel}
  \frD(A) = \{ x \in \frD(A^{*}) \mid \Gamma_0 x = \Gamma_1 x = 0 \}.
\end{equation}
If $\widetilde{A}$ is a self-adjoint operator in a Hilbert space $\widetilde{\frH} \supset \frH$ such that $A \subset \widetilde{A}$, then $\widetilde{A}$ is called a \emph{self-adjoint extension} of $A$. A self-adjoint extension $\widetilde{A}$ is called \emph{minimal} if no nontrivial subspace of $\widetilde{\frH} \ominus \frH$ is reducing for $\widetilde{A}$ \cite[Section 4]{N40}, or equivalently \cite[Lemma 5.17]{S12}
$$
  \widetilde{\frH} = \overline{\linspan \left\{ \left( \widetilde{A} - \lambda I \right)^{-1} x \biggm| x \in \frH,\ \lambda \in \bbC \setminus \bbR \right\}},
$$
where span denotes the linear span.

A holomorphic operator-valued function $R(\lambda)$ on $\bbC \setminus \bbR$ is called a \emph{generalized resolvent} of $A$ if
$$
  R(\lambda) = \left. P_{\frH} \left( \widetilde{A} - \lambda I \right)^{-1} \right|_{\frH}
$$
for some self-adjoint extension $\widetilde{A}$, where $P_{\frH}$ is the orthogonal projection onto $\frH$. For every generalized resolvent there is a unique (up to unitary equivalence) minimal self-adjoint extension with this property (see \cite[Theorem 8]{N40}). On the other hand, there is a one-to-one correspondence between generalized resolvents $R(\lambda)$ and functions $\omega(\lambda)$ holomorphic on the open upper half-plane $\bbC_{+}$ with $|\omega(\lambda)| \le 1$: for every $x \in \frH$ the value $y := R(\lambda) x$ satisfies the equation
$$
  A^{*} y - \lambda y = x
$$
and the condition
\begin{equation*}
  \left( \omega(\lambda) - 1 \right) \Gamma_1 y - \iu \left( \omega(\lambda) + 1 \right) \Gamma_0 y = 0
\end{equation*}
(see \cite{B76}). Denoting
$$
  f(\lambda) := \frac{\iu \omega(\lambda) + \iu}{1 - \omega(\lambda)}
$$
we obtain a one-to-one correspondence between generalized resolvents $R(\lambda)$ and Herglotz--Nevanlinna functions $f$, i.e. functions holomorphic on $\bbC_{+}$ with $\im f(\lambda) \ge 0$ (cf. \cite[Subsection 1.2]{S55}). The above condition then becomes
\begin{equation} \label{eq:boundary}
  \Gamma_1 y + f(\lambda) \Gamma_0 y = 0.
\end{equation}
The ``Dirichlet'' condition $\Gamma_0 y = 0$ corresponds to $f = \infty$. We will denote the generalized resolvent corresponding to $f$ by $R_f(\lambda)$. It should be noted that complete parameterizations of all generalized resolvents in terms of Herglotz--Nevanlinna functions were first obtained independently by Naimark \cite{N43} and Krein \cite{K44}.

We are now ready to construct our minimal self-adjoint extension corresponding to a Herglotz--Nevanlinna function $f$. This function has a unique representation of the form \cite[Appendix F]{S12}
$$
  f(\lambda) = h_0 \lambda + h + \int_{-\infty}^{+\infty} \left( \frac{1}{t - \lambda} - \frac{t}{1 + t^2} \right) \,\du \sigma(t),
$$
where $h_0 \ge 0$, $h \in \bbR$, and
$$
  \int_{-\infty}^{+\infty} \frac{\du \sigma(t)}{1 + t^2} < \infty
$$
(the reader may refer to \cite[Appendix A]{GT00} for some examples of such representations). If $h_0 > 0$ then we consider the Hilbert space $\widetilde{\frH} := \frH \oplus L^2(\mathbb{R}; \du \sigma) \oplus \mathbb{C}$ with inner product given by
$$\langle \widetilde{x}, \widetilde{y} \rangle_{\widetilde{\frH}} := \langle x_0, y_0 \rangle_{\frH} + \int_{-\infty}^{+\infty} x_1(t) \overline{y_1(t)} \,\du \sigma(t) + \frac{x_2 \overline{y_2}}{h_0}$$
for
$$
  \widetilde{x} = \begin{pmatrix} x_0 \\ x_1 \\ x_2 \end{pmatrix}, \qquad \widetilde{y} = \begin{pmatrix} y_0 \\ y_1 \\ y_2 \end{pmatrix} \in \widetilde{\frH},
$$
and define the operator
$$\widetilde{A}\widetilde{x} := \begin{pmatrix} A^{*} x_0 \\ t x_1(t) - \Gamma_0 x_0 \\ \Gamma_1 x_0 + h \Gamma_0 x_0 + \int_{-\infty}^{+\infty} \left( x_1(t) - \frac{t}{1 + t^2} \Gamma_0 x_0 \right) \,\du \sigma(t) \end{pmatrix}$$
with
\begin{equation*}
  \frD(\widetilde{A}) := \{ \widetilde{x} \in \widetilde{\frH} \mid x_0 \in \frD(A^{*}),\ t x_1(t) - \Gamma_0 x_0 \in L^2(\mathbb{R}; \du \sigma),\ x_2 = -h_0 \Gamma_0 x_0 \}.
\end{equation*}
If $h_0 = 0$ then we set $\widetilde{\frH} := \frH \oplus L^2(\mathbb{R}; \du \sigma)$, and define $\widetilde{A}$ by
$$\widetilde{A}\widetilde{x} := \begin{pmatrix} A^{*} x_0 \\ t x_1(t) - \Gamma_0 x_0 \end{pmatrix}$$
and
\begin{multline*}
  \frD(\widetilde{A}) := \left\{ \widetilde{x} \in \widetilde{\frH} \biggm| x_0 \in \frD(A^{*}),\ t x_1(t) - \Gamma_0 x_0 \in L^2(\mathbb{R}; \du \sigma), \vphantom{\int_{-\infty}^{+\infty}} \right. \\
  \left. \Gamma_1 x_0 + h \Gamma_0 x_0 + \int_{-\infty}^{+\infty} \left( x_1(t) - \frac{t}{1 + t^2} \Gamma_0 x_0 \right) \,\du \sigma(t) = 0 \right\}.
\end{multline*}
The extension $\widetilde{A}$ is canonical (i.e. $\widetilde{\frH} = \frH$) if and only if $f$ is a real constant (or $\infty$).

\begin{theorem}
  For each Herglotz--Nevanlinna function $f$ the operator $\widetilde{A}$ defined above is a minimal self-adjoint extension of the operator $A$ with the corresponding generalized resolvent $R_f(\lambda)$.
\end{theorem}
\begin{proof}
We will consider the case $h_0 > 0$; the other case can be proved similarly. To prove the self-adjointness, let $\widetilde{y}$, $\widetilde{z} \in \widetilde{\frH}$ be such that
\begin{equation} \label{eq:adjoint}
  \langle \widetilde{A} \widetilde{x}, \widetilde{y} \rangle_{\widetilde{\frH}} = \langle \widetilde{x}, \widetilde{z} \rangle_{\widetilde{\frH}}
\end{equation}
for all $\widetilde{x} \in \frD(\widetilde{A})$. Taking into account (\ref{eq:kernel}) and setting $x_1(t) \equiv 0 = x_2$ we obtain
$$
  \langle A x_0, y_0 \rangle_{\frH} = \langle x_0, z_0 \rangle_{\frH}, \qquad x_0 \in \frD(A).
$$
Hence $y_0 \in \frD(A^{*})$ and $z_0 = A^{*} y_0$. By surjectivity of $\Gamma$, there exists
$x_0 \in \frD(A^{*})$ with $\Gamma_0 x_0 = 0$ and $\Gamma_1 x_0 = 1$. Then for this $x_0$ and $x_1(t) \equiv 0$ the equality (\ref{eq:adjoint}) gives $y_2 = -h_0 \Gamma_0 y_0$. Setting $x_0 = 0$ in (\ref{eq:adjoint}) we get
$$
  \int_{-\infty}^{+\infty} x_1(t) \overline{\left( t y_1(t) - \Gamma_0 y_0 - z_1(t) \right)} \,\du \sigma(t) = 0
$$
for all $x_1(t) \in L^2(\mathbb{R}; \du \sigma)$ with $t x_1(t) \in L^2(\mathbb{R}; \du \sigma)$. Since such functions are dense in $L^2(\mathbb{R}; \du \sigma)$, we obtain $t y_1(t) - \Gamma_0 y_0 = z_1(t) \in L^2(\mathbb{R}; \du \sigma)$. Finally choosing $x_0 \in \frD(A^{*})$ with $\Gamma_0 x_0 = 1$, the equality (\ref{eq:adjoint}) yields
$$
  z_2 = \Gamma_1 y_0 + h \Gamma_0 y_0 + \int_{-\infty}^{+\infty} \left( y_1(t) - \frac{t}{1 + t^2} \Gamma_0 y_0 \right) \,\du \sigma(t).
$$
Thus $\widetilde{y} \in \frD(\widetilde{A})$ and $\widetilde{A} \widetilde{y} = \widetilde{z}$.

To see that the generalized resolvent corresponding to $\widetilde{A}$ is $R_f(\lambda)$, it suffices to check that if
$$
  \widetilde{y} = \begin{pmatrix} y_0 \\ y_1 \\ y_2 \end{pmatrix} = \left( \widetilde{A} - \lambda I \right)^{-1} \begin{pmatrix} x \\ 0 \\ 0 \end{pmatrix}
$$
then $y_0$ satisfies (\ref{eq:boundary}), and this is straightforward. Finally, to check the minimality, we need to verify that the linear span of all $\widetilde{y}$ of this form with all possible values of $x \in \frH$ and $\lambda \in \bbC \setminus \bbR$ is dense in $\widetilde{\frH}$. To this end, let $\widetilde{z} \in \widetilde{\frH}$ be orthogonal to all such $\widetilde{y}$, i.e.
$$
  \langle \widetilde{z}, \widetilde{y} \rangle_{\widetilde{\frH}} = \langle z_0, y_0 \rangle_{\frH} + \int_{-\infty}^{+\infty} z_1(t) \frac{\overline{\Gamma_0 y_0}}{t - \overline{\lambda}} \,\du \sigma(t) - z_2 \overline{\Gamma_0 y_0} = 0
$$
for all $y_0 \in \frD(A^{*})$. Surjectivity of $\Gamma$ implies $z_0 = 0$ and
$$
  \int_{-\infty}^{+\infty} \frac{z_1(t)}{t - \overline{\lambda}} \,\du \sigma(t) = z_2, \qquad \lambda \in \bbC \setminus \bbR.
$$
Now the Stieltjes--Perron inversion formula \cite[Theorem F.2]{S12} applied to the regular complex Borel measure $z_1\,\du \sigma$ yields $z_1(t) = 0$ for $\du \sigma$-a.e. $t$ and consequently $z_2 = 0$.
\end{proof}

\end{document}